\documentclass[10pt,leqno]{amsart}

\usepackage{amsmath}
\usepackage{amsfonts}
\usepackage{amssymb}
\usepackage{graphicx}
\usepackage{color}
\usepackage{hyperref}
\usepackage{xcolor}

\newtheorem{theorem}{Theorem}[section]

\newtheorem*{theorem*}{Theorem}
\newtheorem{lemma}{Lemma}[section]

\newtheorem{proposition}{Proposition}[section]

\newtheorem{definition}[theorem]{Definition}

\newtheorem{remark}[theorem]{Remark}

\def \b {\beta}

\def\Ric{\text{Ric}}

\def\th{\theta}

\def\ve{\varepsilon}
\def\p{\partial}

\def\R{\mathbb{R}}

\def\vp{\varphi}

\def\k{\kappa}

\def\tm{\tilde{\mu}}

\def\sn{\operatorname{sn}}

\def\Ric{\operatorname{Ric}}

\def\n{\nabla}

\numberwithin{equation}{section}

\begin{document}


\title[Comparison results on manifolds]{Comparison results for Poisson equation with mixed  boundary condition on manifolds}

\author{Haiqing Cheng}
\address{School of Mathematical Sciences, Soochow University, Suzhou, 215006, China}
\email{chq4523@163.com}

\author{Tengfei Ma}
\address{Department of Mathematics, Nanjing University, Nanjing, 210093, China}
\email{math\_mtf@163.com}

\author{Kui Wang} \thanks{The research of the third author is supported by NSFC No.11601359} 
\address{School of Mathematical Sciences, Soochow University, Suzhou, 215006, China}
\email{kuiwang@suda.edu.cn}




\begin{abstract}
In this article,  we establish a $L^1$ estimate for solutions to Poisson equation with mixed  boundary condition, on complete noncompact manifolds with nonnegative Ricci curvature and compact manifolds with positive Ricci curvature respectively.  On Riemann surfaces we obtain a Talenti-type comparison. Our results generalize  main theorems in \cite{ACNT21} to Riemannian setting,  and  Chen-Li's result \cite{CL21}  to the case of variable Robin  parameter. 
\end{abstract}

\maketitle

\section{Introduction}
Let $\Omega\subset\R^n$ be a bounded domain with nonempty smooth boundary, $\Omega^\sharp\subset\R^n$ be a round ball  with the same volume as $\Omega$, $f(x)$ be a nonnegative function on $\Omega$ and $f^\sharp$ be the Schwarz decreasing rearrangement of $f$ (see Definition \ref{sch}). Assume $u(x)$ and $v(x)$ are solutions to
$$
-\Delta u(x) =f(x), \qquad x\in \Omega,
$$
and 
$$
-\Delta v(x) =f^\sharp (x), \qquad x\in \Omega^\sharp
$$
with  Dirichlet boundary, respectively.  Talenti \cite{Ta76} proved that
\begin{align}\label{1.1}
    u^\sharp (x)\le v(x), \qquad x\in\Omega^\sharp. \end{align}
Moreover the equality occurs if and only if $f(x)$ is a radial function and $\Omega$ is a round ball, see \cite{Ke88}. The key tools of the proof are Schwarz  symmetrization and isoperimetric inequalities on manifolds.
Talenti's comparison \eqref{1.1} plays  important roles in both  partial differential equations and geometry problems.  There provide $L^\infty$ estimates for solutions to PDEs and Faber-Krahn type inequality for the first Dirichlet  eigenvalue, see  \cite{CL21, HMP16, Ke88}. Talenti's comparison was generalized to nonlinear elliptic  and parabolic equations with Dirichlet boundary condition (cf. \cite{ATL90, AFTL97, Ba76,Ta79} and references therein),  to compact Riemannian manifolds with positive Ricci curvature  \cite{CLM18}, and to noncompact manifolds with nonnegative Ricci curvature and positive asymptotic volume ratio \cite{CL21}. We also refer the reader to the excellent books \cite{Ka85, Ke06} for related topics.
 
Recently, Alvino, Nitsch and Trombetti \cite{ANT21}  studied the Poisson equation with Robin boundary condition when Robin parameter is a positive constant. They proved  estimate \eqref{1.1}  on planer domains, and a sharp $L^1$ estimate for  higher dimensions.
 These results were generalized to Riemannian manifolds \cite{CLY21},  to Robin boundary with variable Robin parameter \cite{ACNT21}, and to the
torsion problem for the Hermite operator \cite{CGNT21,DNST21}. 

The purpose of the present paper is to study  Talenti's comparison and related estimates for solutions to Poisson equation with mixed boundary condition on manifolds. In particular,  we generalize  Alvino-Chiacchio-Nitsch-Trombetti's result \cite{ACNT21} to compact manifolds with positive Ricci curvature, and to noncompact manifolds with nonnegative Ricci curvature.

To state our results, we give some notations. Let $(M, g)$ be an  $n$-dimensional complete Riemannian manifold, which is either compact with $\Ric\ge (n-1)\k$ for $\k>0$, or noncompact with nonnegative Ricci curvature and positive asymptotic volume ratio. Denote by 
\begin{align}\label{1.2}
    \theta=
    \begin{cases}
    \lim_{r\to \infty}\frac{|B(r)|}{\omega_n r^n}, & \quad \k=0,\\
    \frac{|M|}{|M_\k|}, &\quad \k>0,
    \end{cases}
\end{align}
where  $B(r)$ is a round geodesic ball of radius $r$ in $M$, $M_\k$ is the $n$ dimensional space form of  sectional curvature $\k$, $\omega_n$ is the volume of the unit ball in $\R^n$, and $|M|$ denotes the volume of $M$. It follows from  the Bishop-Gromov volume comparison that $\th\le 1$.
Let $\Omega\subset M$ be a bounded domain with nonempty smooth boundary, $f(x)$ be a  nonnegative smooth function not identically zero on $\Omega$,
and $\b(x)$ be a positive smooth  function on $\p \Omega$. We consider the following Poisson equation with mixed boundary condition
\begin{align}\label{1.3}
    \begin{cases}
    -\Delta u(x)= f(x), &\quad x\in \Omega,\\
    \frac{\p u}{\p \nu}+\b(x) u=0, & \quad x\in \p \Omega,
    \end{cases}
\end{align}
where $\Delta$ denotes the Laplace-Beltrami operator and $\nu$ denotes the outward unit normal to $\p \Omega$. Since $\b(x)>0$ and $f\ge 0$, it then follows easily from  maximum principle that the solution to \eqref{1.3} is positive on $\Omega$, see Proposition \ref{po-u}.  Let  $\Omega^{\sharp}$ be a round geodesic ball in $M_\k$ satisfying $\theta |\Omega^\sharp|=|\Omega|$, 
 $f^\sharp$ defined on $\Omega^\sharp$ be the Schwarz decreasing rearrangement of $f$, and $v(x)$ be the solution to the Schwarz decreasing rearrangement equation of \eqref{1.3}, namely
 \begin{align}\label{1.4}
    \begin{cases}
    -\Delta v(x)= f^\sharp (x), &\quad x\in \Omega^\sharp,\\
    \frac{\p v}{\p \nu}+\bar \b v=0, & \quad x\in \p \Omega^\sharp,
    \end{cases}
\end{align}
where $\bar \b$ is a positive constant defined by 
\begin{align}\label{1.5}
 \bar \b= \frac{\th |\p \Omega^\sharp|}{\int_{\p \Omega}\frac{1}{\b(x)}\, dA}.
\end{align}
Throughout the paper, $dA$  denotes the induced measure on boundary $\p \Omega$  or $\p \Omega^\sharp$.
Since the boundary value problem \eqref{1.4} is radially symmetric, then the solution $v(x)$ is radially symmetric as well. Moreover $v(x)$ is monotone decreasing along the radial direction, see Proposition \ref{prop2} below.

Our first result is concerning a $L^1$ comparison between $u(x)$ and $v(x)$. Precisely we prove
\begin{theorem}\label{thm1}
Let $(M, g)$ be an  $n$-dimensional complete Riemannian manifold, which is either compact with $\Ric\ge (n-1)\k$ for $\k>0$, or noncompact with nonnegative Ricci curvature and positive asymptotic volume ratio. 
 Let $\Omega\subset M$ be a bounded domain with nonempty smooth boundary, $f(x)$ be a nonnegative smooth  function not identically zero on $\Omega$,
and $\b(x)$ be a positive smooth  function on $\p \Omega$. Let $u(x)$ and $v(x)$ be the solutions to equations \eqref{1.3} and \eqref{1.4} respectively. If $n\ge 3$, we assume further that for all measurable $E\subset\Omega$, it holds
\begin{align}\label{1.6}
    \int_E f(x)\, dx\le \frac{|E|^{\frac{n-2}{n}}}{|\Omega|^{\frac{n-2}{n}}}\int_\Omega f(x)\, dx.
\end{align}
Then 
\begin{align}\label{1.7}
||u||_{L^1(\Omega)}\le \th ||v||_{L^1(\Omega^\sharp)}.
\end{align}
\end{theorem}

\begin{remark}
In this paper, we mainly focus on $L^1$ estimate \eqref{1.7} on manifolds and we then set up Theorem \ref{thm1} in  smooth case. In fact, Theorem \ref{thm1}  remains valid under the assumptions that $\Omega$ is a bounded domain with  Lipschitz boundary, $f(x)\in L^2(\Omega)$ and $\b(x)$ is a positive measurable function defined on $\p \Omega$, see \cite{ACNT21}.
\end{remark}

When $n=2$ and $f$ is a constant,  we prove a Talenti-type comparison, similarly as in \cite{ACNT21}. 
\begin{theorem}\label{thm2}
Let $(M, g)$ be a  complete Riemann surface, which is either compact with $\Ric\ge \k>0$, or noncompact with $\Ric\ge 0$ and positive asymptotic volume ratio. Let $u(x)$ and $v(x)$ be the solutions to equations \eqref{1.3} and \eqref{1.4} respectively. Suppose $f(x)\equiv 1$ and $\beta(x)>0$. Then
\begin{align}\label{1.8}
u^\sharp(x)\le v(x), \quad x \in \Omega^\sharp.
\end{align}
Moreover, the equality case of \eqref{1.8} occurs if and only if $M$ and $\Omega$ are isometric to $M_\k$ and $\Omega^\sharp$ respectively,  and $\b(x)$ is a constant on $\p \Omega$.
\end{theorem}

The rest of this paper  is  organized  as  follows.   In  Section  \ref{sect2}, we  recall the Schwarz decreasing  rearrangement and isoperimetric inequalities on manifolds, and give some properties on solutions to Poisson equation with Robin boundary condition.  In Section \ref{sect3}, we prove Theorem 1.1 and Theorem 1.2.

\section{Preliminaries}\label{sect2}
\subsection{Schwarz decreasing rearrangements.} 
Let $(M,g)$ be an  $n$-dimensional complete Riemannian manifold with $\Ric\ge (n-1)\k$, which is either  noncompact with $\k=0$ and positive asymptotic volume ratio, or compact with $\k>0$. Let $M_\k$ be the $n$-dimensional space form of constant sectional curvature $\k$. 
 Let $\Omega$ be a bounded  domain in $M$ and $\Omega^\sharp$ be a round geodesic ball in $M_\k$ with volume 
$|\Omega|/\th$, where $\th\in(0,1]$ is  a constant defined by \eqref{1.2}.

We recall the definitions and properties of the Schwarz decreasing rearrangement of non-negative functions on manifolds, see also  \cite[Section 2]{CL21} and \cite[Section 2]{CLM18}.
\begin{definition}\label{sch}
Let $h(x)$ be a nonnegative measurable function on $\Omega$. Denote by 
$\Omega_{h,t}=\{x\in \Omega: h(x)>t\}$
and 
$
\mu_h(t)=|\Omega_{h,t}|
$,
the decreasing rearrangement $h^*$ of $h$ is defined by
\begin{align}\label{2.1}
    h^*(s)=\begin{cases}
    \operatorname{ess} \sup_{x\in \Omega} \, h(x), & \quad s=0,\\
    \inf\{t\ge 0: \mu_h(t)<s\}, & \quad s>0,
    \end{cases}
\end{align}
for $s\in [0, |\Omega|]$.  The Schwarz decreasing rearrangement of $h(x)$ is defined by 
\begin{align}\label{2.2}
    h^\sharp(x)=h^*\big(\th \omega_n r(x)^n\big), \quad x\in \Omega^\sharp,
\end{align}
where $r(x)$ is the distance function from the center of $\Omega^\sharp$ in $M_\k$, and $\omega_n$ is the volume of unit ball in $\R^n$.
\end{definition}
It follows from \eqref{2.1} and \eqref{2.2} that
\begin{align}\label{2.3}
    \mu_h(t)=\th \mu_{h^\sharp}(t)
\end{align}
for $t\ge 0$. The Fubini's theorem gives that for   $h\in L^p(\Omega)$ and $p\ge 1$ it holds
\begin{align*}
    \int_\Omega h^p(x)\, dx=\int_0^{|\Omega|} (h^*)^p(s)\, ds=\th \int_{\Omega^\sharp} (h^\sharp)^p(x)\, dx.
\end{align*}
For any nonnegative functions $f(x)$ and $g(x)$,  the Hardy-Littlewood inequality gives
\begin{align}\label{2.4}
    \int_\Omega f(x) g(x)\,dx\le \int_0^{|\Omega|} f^*(s) g^*(s) \, ds,
\end{align}
and taking $g(x)$  as the characteristic function of $\Omega_{h,t}$ in above inequality yields
\begin{align}\label{2.5}
\int_{\Omega_{h,t}} f(x)\,dx\le \int_0^{\mu_h(t)} f^*(s)\,ds.    
\end{align}

\subsection{Isoperimetric inequalities on manifolds.} To prove Theorem \ref{thm1} and Theorem \ref{thm2}, we require the following isoperimetric inequality on manifolds with Ricci curvature bounded from below.
 \begin{theorem}
Let $M$, $\th$, $\Omega$ and $\Omega^\sharp$ be the same as in Theorem \ref{thm1}. It holds
\begin{align}\label{2.6}
    |\p \Omega|\ge \theta  |\p \Omega^\sharp|,
\end{align}
where $|\p \Omega|$ denotes the $(n-1)$-dimensional area of $\p \Omega$. Moreover the equality case of \eqref{2.6} occurs if and only if $\Omega$ is isometric to $\Omega^\sharp$.
 \end{theorem}
When $\k>0$,  inequality \eqref{2.6} was known as L\'evy-Gromov  isoperimetric inequality \cite{Gr99}, 
see also \cite[Theorem 2.1]{NW16}. When $M$ is noncompact and $\k= 0$,  inequality \eqref{2.6} was proved  by Agostiniani, Fogagnolo and Mazzieri \cite{AFM20} when $n=3$, and  by Brendle \cite{Bre21} for all dimensions.  Brendle  proved in \cite{Bre21} that the isoperimetric inequality \eqref{2.6} also holds true when $\Omega$ is an $n$-dimensional compact minimal submanifold of  $M$ of dimension $n+2$  with nonnegative sectional curvature as well.

\subsection {Poisson equation with Robin boundary} 

In this subsection, we collect some known
facts about  solutions to Poisson equation with Robin boundary condition.
\begin{proposition}\label{po-u}
Let $u(x)$ be the solution to \eqref{1.3}. Assume $f(x)$ is nonnegative and  not identically zero on $\Omega$, and $\b(x)>0$ on $\p \Omega$. Then we have
\begin{align*}
    u(x)>0
\end{align*}
for all $x\in \Omega$.
\end{proposition}
\begin{proof}
Letting $\displaystyle \bar u(x)=\frac 1 2 (u(x)-|u(x)|)$ and using integration by parts, we estimate that
\begin{align}\label{2.7}
    \int_\Omega \langle \n \bar u, \n u\rangle\, dx=&\int_{\p \Omega}\bar u \frac{\p u}{\p \nu}\, dA-\int_{\Omega} \bar u \Delta u\, dx\\
    =&-\int_{\p \Omega} \b u\bar u\, dA+\int_\Omega \bar u f\, dx\nonumber \\
    \le & 0,\nonumber
\end{align}
where we used  equation \eqref{1.3} in the second equality and nonnegativity of $f$ in the inequality.  Observing that $$\int_\Omega \langle \n \bar u,\n u\rangle\, dx=\int_\Omega |\n \bar u|^2\, dx, $$ 
then we have
$$
\int_\Omega |\n \bar u|^2\, dx=0.
$$
So $\bar u(x)$ is a constant on $\Omega$ and inequality \eqref{2.7} holds as an equality, hence $\bar u(x)=0$, equivalently $u(x)\ge 0$. Since $u(x)$ is a supharmonic function by equation \eqref{1.3} and $f$ is not identically zero, then we conclude from strong maximum principle that $u(x)>0$ on $\Omega$.
\end{proof}
\begin{proposition}\label{prop2}
Let $v(x)$ be the solution to \eqref{1.4}, rewritten as  $\vp(r(x))$. If $f(x)$ is non-negative and not identically zero on $\Omega$,  and $\bar \b>0$, then
\begin{align}\label{2.8}
    \vp'(r)< 0
\end{align}
for $r\in(0, R_0)$. Where $\Omega^\sharp=B(R_0)$, and $R_0\le \frac{\pi}{\sqrt{\k}}$ if $\k>0$ in  view of the Myers theorem.
\end{proposition}
\begin{proof}
We rewrite $f^\sharp(x)$ as $h(r(x))$ and denote by 
$$
\sn_\k(r)=\begin{cases}
\frac{\sin \sqrt{\k}r}{\sqrt{\k}}, & \k>0,\\
r, &\k=0,
\end{cases}
$$
then equation \eqref{1.4} is equivalent to
$$
\vp''(r)+(n-1)\frac{\sn_\k'(r)}{\sn_\k (r)}\vp'(r)=-h(r).
$$
Since  $h(r)\ge 0$, we then have
$$
(\sn_\k^{n-1}(r)\vp'(r))'=-\sn_\k^{n-1}(r) h(r)\le 0
$$
for $r\in(0,R_0)$. 
Since $h(r)$ is monotone nonincreasing in $(0,R_0)$ and not identically zero, then there exists a $\ve_0<R_0$ such that
$h(r)>0$ for $r<\ve_0$. Hence $\vp'(r)<0$ for all $r\in (0, R_0)$.
\end{proof}

\section{Proofs of Theorem \ref{thm1} and Theorem \ref{thm2}}\label{sect3}
In this section, we will prove the main theorems.
For simplicity, we rewrite $\Omega_{u,t}$  as $\Omega_t$, and $\Omega^\sharp_{v,t}$ as $\Omega^\sharp_t$ for short.  For $s\in (0,|\Omega^\sharp|)$, we denote
\begin{align*}
  a(s)=\frac{s^\frac{n-1}{n}}{|\p B_s|},   
 \end{align*} 
 where $B_s$ is a round geodesic ball in $M_\k$ with volume $s$. It is easily checked that $a(s)=n^{-1}\omega_n^{-1/n}$ if $\k=0$, and $a(s)$ is monotone increasing in $s$ if $\k>0$.  Using isoperimetric inequality \eqref{2.6}, we prove the following lemma,  which  will be used later. 
\begin{lemma}
Under the hypotheses of  Theorem \ref{thm1}, we have
\begin{align}\label{3.1}
\th^2 (\frac{\mu_u(t)}{\th})^{\frac{2n-2}{n}}\le  a^2\big(\frac{\mu_u(t)}{\th}\big)\Big(-\mu_u'(t)+ \int_{\p\Omega_t\cap \p \Omega} \frac{1}{\b (x) u(x)}\, dA\Big)\int_{0}^{\mu_u(t)} f^*(s)\, ds,
\end{align}
and 
\begin{align}\label{3.2}
 \mu_v^{\frac{2n-2}{n}}(t)= a^2\big(\mu_v(t)\big) \Big(-\mu_v'(t)+\int_{\p \Omega^\sharp_t\cap \p \Omega^\sharp} \frac{1}{\bar \b v(x)}\, dA\Big)\int_{0}^{\mu_v(t)} (f^\sharp)^*(s)\, ds
\end{align}
 for a.e. $t>0$.
\end{lemma}
\begin{proof}
By  the Morse-Sard theorem, we have
\begin{align*}
    \p \Omega_{t}=\{x\in \Omega: u(x)=t\}\cup\{x\in \p \Omega: u(x)\ge t\}
\end{align*}
for a.e. $t>0$. Let
\begin{align*}
    g(x)=\begin{cases}
    |\n u|, & x\in \p \Omega_t\cap \Omega,\\
    -\frac{\p u}{\p \nu}, & x\in \p \Omega_t\cap \p \Omega.
    \end{cases}
\end{align*}
Applying the divergence theorem and equation \eqref{1.3}, we derive that
\begin{align} \label{3.3}
    \int_{\p\Omega_t} g(x)\, dA=\int_{\p\Omega_t} -\frac {\p u}{\p \nu}\, dA=-\int_{\Omega_t} \Delta u \,dx=\int_{\Omega_t} f(x)\, dx\le \int_0^{\mu_u(t)} f^*(s)\, ds,
\end{align}
where we have used  \eqref{2.5} in the inequality.
Using the H\"older inequality, we estimate 
\begin{align}\label{3.4}
 |\p \Omega_t|^2
 \le & \int_{\p \Omega_t} g(x)\, dA \int_{\p \Omega_t}\frac 1 {g(x)}\, dA\\
 \le &\int_0^{\mu_u(t)} f^*(s)\, ds \Big(\int_{\p\Omega_t\cap \Omega} \frac 1 {|\n u|}\, dA+\int_{\p\Omega_t\cap \p \Omega} \frac{1}{\b (x) u(x)}\, dA\Big)\nonumber\\
= &\int_0^{\mu_u(t)} f^*(s)\, ds \Big(-\mu_u'(t)+ \int_{\p\Omega_t\cap \p \Omega} \frac{1}{\b(x)   u(x)}\, dA\Big),\nonumber
\end{align}
where we have used inequality \eqref{3.3} in the second inequality and the coarea formula in the equality.

On the other hand,  the isoperimetric inequality \eqref{2.6} yields
\begin{align}\label{3.5}
    |\p \Omega_t|\ge \th |\p (\Omega_t^\sharp)|=\frac{\th }{a(\mu_u(t)/\th)}(\frac{\mu_u(t)}{\theta})^{\frac{n-1}{n}}.
\end{align}
Assembling inequalities \eqref{3.4} and \eqref{3.5}, we get
\begin{align*}
\frac{\th^2 }{a^2(\mu_u(t)/\th)}\mu_u^{\frac{2n-2}{n}}(t)
 \le \int_0^{\mu_u(t)} f^*(s)\, ds \Big(-\mu_u'(t)+ \int_{\p\Omega_t\cap \p \Omega} \frac{1}{\b(x)   u(x)}\, dA\Big),
\end{align*}
which proves inequality \eqref{3.1}. 

If  $v(x)$ is the solution to equation \eqref{1.4}, $v(x)$ is a radial function on $\Omega^\sharp$ and decreasing along the radial direction by inequality \eqref{2.8}, and  $\Omega^\sharp_t$ is a round ball. Therefore all previous inequalities hold as equalities if we replace $u(x)$ by $v(x)$, hence
\begin{align*}
\mu_v^{\frac{2n-2}{n}}(t)=& a^2(\mu_v(t)) \Big(-\mu_v'(t)+ \int_{\p \Omega^\sharp_t\cap\p \Omega^\sharp} \frac{1}{\bar \b v(x)}\, dA\Big)\int_{0}^{\mu_v(t)} (f^\sharp)^*(s)\, ds
\end{align*}
for all $t>0$. Thus we complete the proof of the lemma.
\end{proof}
\begin{lemma}
Suppose $u$ and $v$ are solutions to \eqref{1.3} and \eqref{1.4}. Then both $u$ and $v$ attain their minima on $\p \Omega$ and $\p \Omega^\sharp$. Moreover if we denote by $u_0$ and $v_0$ the minima of $u$ and $v$ respectively, then
\begin{align}\label{3.6}
    u_0\le v_0.
\end{align}
\end{lemma}
\begin{proof}
Recall that $f$ is nonnegative, then $-\Delta u\ge 0$ and $-\Delta v\ge 0$. Therefore $u$ and $v$ attain their minima on $\p \Omega$ and $\p \Omega^\sharp$. Moreover  $u(x)>u_0$ for $x\in \Omega$, $v(x)>v_0$ for $x\in \Omega^\sharp$ unless $f$ is identically zero.  

Noting that $v(x)=v_0$ on $\p \Omega^\sharp$ and using integration by parts we compute  that
\begin{align*}
    v_0|\p \Omega^\sharp|^2=&\int_{\p \Omega^\sharp} \frac{1}{\bar \b} \,dA\int_{\p \Omega^\sharp}\bar\b v(x)\, dA\\
    =&\int_{\p \Omega^\sharp} \frac{1}{\bar \b} \,dA\int_{\Omega^\sharp} -\Delta v\,dx\\
    =&\frac 1 \th \int_{\p \Omega} \frac{1}{\b(x)} \,dA \int_{\Omega^\sharp} f^\sharp(x) \,dx\\
      =&\frac 1 {\th^2} \int_{\p \Omega} \frac{1}{\b(x)} \,dA \int_{\Omega} -\Delta u(x) \,dx\\
    =&\frac 1 {\th^2}\int_{\p \Omega} \frac{1}{\b(x)} \,dA \int_{\p \Omega} \b(x) u(x) \, dA,
\end{align*}
where we have used equations \eqref{1.3},   \eqref{1.4}, and equality \eqref{1.5}. By the H\"older inequality, it holds
$$
(\int_{\p \Omega} \sqrt{u(x)}\, dA)^2\le \int_{\p \Omega} \frac{1}{\b(x)} \,dA \int_{\p \Omega} \b(x) u(x) \, dA,
$$
then we have
\begin{align*}
    v_0|\p \Omega^\sharp|^2
    \ge \frac 1 {\th^2}(\int_{\p \Omega} \sqrt{u(x)} \,dA)^2\ge  \frac{u_0}{\th^2}|\p \Omega|^2
   \ge  u_0|\p \Omega^\sharp|^2,
\end{align*}
where we used isoperimetric inequality \eqref{2.6} in the last inequality. Thus we have $u_0\le v_0$, proving \eqref{3.6}. 
\end{proof}
Now we turn to prove Theorem \ref{thm1}.
\begin{proof}[Proof of Theorem \ref{thm1}]
Set 
$
\tm_u(t)=\frac{\mu_u(t)}{\th}
$.
Dividing inequality \eqref{3.1} by $\tm_u(t)^{\frac{2}{n}-1}$ and  integrating over $[0,\tau]$, we obtain
\begin{align*}
\theta^2 \int_{0}^{\tau} \tilde{\mu}_u(t) \,d t \leq & \int_{0}^{\tau}\left( a^2(\tilde{\mu}_u(t))\tilde{\mu}_u(t)^{\frac{2}{n}-1}\left(\int_{0}^{\mu_u(t)} f^{*}(s)\, d s\right)\left(-\mu_u^{\prime}(t)\right)\right)\, d t \\
&+\int_{0}^{\tau}  a^2(\tilde{\mu}_u(t))\tilde{\mu}_u(t)^{\frac{2}{n}-1} \int_{0}^{\mu_u(t)} f^{*}(s)\, d s\left(\int_{\p \Omega_t\cap \p \Omega} \frac{1}{\beta(x)} \frac{1}{u(x)} \, d A\right) d t \\
\leq &
\theta^2\int_{0}^{\tau}\left( a^2(\tilde{\mu}_u(t))\tilde{\mu}_u(t)^{\frac{2}{n}-1}
\int_{0}^{\tilde{\mu}_u(t)} (f^\sharp)^{*}(s) \, d s\right)(-d \tilde{\mu}_u(t)) \\
&+\th a^2(|\Omega^\sharp|)|\Omega^\sharp|^{\frac{2}{n}-1} \int_{0}^{|\Omega^\sharp|} (f^{\sharp})^{*}(s)\, d s \int_{0}^{\tau}\left(\int_{\p \Omega_t\cap \p \Omega} \frac{1}{\beta(x)} \frac{1}{u(x)} d A\right)\, d t,
\end{align*}
where we have used 
$$
\mu_u(t)^{\frac 2 n -1}\int_0^{\mu_u(t)} f^*(s)\, ds\le  |\Omega|^{\frac 2 n -1}\int_{\Omega} f(x)\,dx=|\Omega|^{\frac 2 n -1}\int_0^{|\Omega|} f^*(s)\,ds,
$$
which follows from the assumption \eqref{1.6} and the fact $a(s)$ is monotone increasing in $s$. Applying Fubini's theorem, we estimate
\begin{align}\label{3.7}
   \int_0^\tau \int_{\p \Omega_t\cap \p \Omega} \frac{1}{\b (x) u(x)}\, dA\, dt\le   \int_0^\infty \int_{\p \Omega_t\cap \p \Omega} \frac{1}{\b (x) u(x)}\, dA\, dt=\int_{\p \Omega}\frac 1 {\b(x)} \,dA,
\end{align}
then we conclude
\begin{align}\label{3.8}
\theta^2 \int_{0}^{\tau} \tilde{\mu}_u(t)\, d t \le& -\th^2 F(\tm_u(\tau))+\th^2 F(\tm_u(0))\\
&+\theta a^2(|\Omega^\sharp|)|\Omega^\sharp|^{\frac{2}{n}-1} \int_{0}^{|\Omega^\sharp|} (f^{\sharp})^{*}(s) \, d s\int_{\p \Omega}\frac 1 {\b(x)} \,dA,\nonumber
\end{align}
where $F(s)$ is defined by
$$F(s):=\int_{0}^{s} a^2(\sigma)\sigma^{\frac{2}{n}-1} \int_{0}^{\sigma} (f^\sharp)^{*}(r) \,d r d \sigma.$$
Substituting   \eqref{1.5} into  inequality \eqref{3.8} yields
\begin{align}\label{3.9}
\int_{0}^{\tau} \tilde{\mu}_u(t) \,d t \le - F(\tm_u(\tau))+ F(\tm_u(0))+a^2(|\Omega^\sharp|)|\Omega^\sharp|^{\frac{2}{n}-1} \int_{0}^{|\Omega^\sharp|} (f^{\sharp})^{*}(s) \, d s\int_{\p \Omega^\sharp}\frac 1 {\bar \b} \,dA.   
\end{align}
For $\tau>v_0$,  we have $\Omega^\sharp_\tau\subset\subset \Omega^\sharp $, and then equality \eqref{3.2} becomes to
\begin{align*}
 \mu_v^{\frac{2n-2}{n}}(t)= a^2\big(\mu_v(t)\big) \Big(-\mu_v'(t)\Big)\int_{0}^{\mu_v(t)} (f^\sharp)^*(s)\, ds,
\end{align*}
implying 
\begin{align}\label{3.10}
 \int_{0}^{\tau} \mu_v(t) \, d t+F(\mu_v(\tau))= \int_{0}^{v_0} \mu_v(t) \, d t+F(\mu_v(v_0)).  
\end{align}
 Dividing equation \eqref{3.2} by $\mu_v(t)^{\frac{2-n}{n}}$ and integrating over $[0, v_0]$ yields
$$
\int_{0}^{v_0} \mu_v(t) \, d t+F(\mu_v(v_0))=F(\mu_v(0))+a^2(|\Omega^\sharp|)|\Omega^\sharp|^{\frac{2}{n}-1} \int_{0}^{|\Omega^\sharp|} (f^{\sharp})^{*}(s) \, d s\int_{\p \Omega^\sharp}\frac 1 {\bar \b} \,dA, $$
and substituting  above equality to \eqref{3.10}, we obtain for all $\tau>v_0$ that
\begin{align}\label{3.11}
\int_{0}^{\tau} \mu_v(t) \,d t+F(\mu_v(\tau))=&F(\mu_v(0))+a^2(|\Omega^\sharp|)|\Omega^\sharp|^{\frac{2}{n}-1} \int_{0}^{|\Omega^\sharp|} (f^{\sharp})^{*}(s) \, d s\int_{\p \Omega^\sharp}\frac 1 {\bar \b} \,dA.
\end{align}
Combining  inequality \eqref{3.9} and  equality \eqref{3.11}, we  get
\begin{align}\label{3.12}
\int_{0}^{\tau} \tm_u(t) \, d t-\int_{0}^{\tau} \mu_v(t) \, d t \le & -F(\tilde{\mu}_u(t))+F(\mu_v(t))+F(\tilde{\mu}_u(0))-F(\mu_v(0))\\
=& -F(\tilde{\mu}_u(t))+F(\mu_v(t)),\nonumber
\end{align}
where in the equality we used that
\begin{align*}
    \tilde{\mu}_u(0)=\frac{|\Omega|}{\theta}=|\Omega^\sharp|=\mu_v(0).
\end{align*}
Letting $\tau \to \infty$ in \eqref{3.12}, we get
$$
\int_0^\infty \mu_u(t)\, dt\le \int_0^\infty \th \mu_v(t)\, dt,
$$ 
 where we have used $\displaystyle \lim_{\tau \to \infty} \mu_u(\tau)=\lim_{\tau\to \infty}\mu_v(\tau)=0$. Hence \eqref{1.7} holds true.
\end{proof}

Now we deal with the case $n=2$ and $f(x)\equiv 1$.
\begin{proof}[Proof of Theorem \ref{thm2}]
When $n=2$ and $f(x)\equiv 1$, inequality \eqref{3.1}  gives
\begin{align}\label{3.13}
\th \le & \Big(-\mu_u'(t)+\int_{\p \Omega_t \cap \p \Omega}\frac 1 {\b(x)u(x)}\, dA\Big) a^2(\mu_u(t)/\theta)\\
\le & \Big(-\mu_u'(t)+\int_{\p \Omega_t \cap \p \Omega}\frac 1 {\b(x)u(x)}\, dA\Big) a^2(|\Omega^\sharp|),\nonumber 
\end{align}
where we have used the fact that $a(s)$ is monotone increasing in $s$ in the last inequality. 
Integrating  inequality \eqref{3.13}  over $[0, \tau]$ yields
\begin{align}\label{3.14}
\th \tau \le & \Big(-\mu_u(\tau) +|\Omega| +\int_0^\tau \int_{\p \Omega_t\cap \p \Omega}\frac 1 {\b(x) u(x)}\,dA\,dt\Big)a^2(|\Omega^\sharp|)\\
 \le & \Big(-\mu_u(\tau) +|\Omega| +\int_{\p \Omega}\frac 1 {\b(x)} \,dA\Big)a^2(|\Omega^\sharp|),\nonumber
\end{align}
where we used \eqref{3.7} in the last inequality.

Analogously, by using equality \eqref{3.2} we have
\begin{align}\label{3.15}
  \tau = \Big(-\mu_v(\tau) +|\Omega^\sharp| +\int_{\p \Omega^\sharp}\frac 1 {\bar \b} \,dA\Big)a^2(|\Omega^\sharp|)
\end{align}
for $\tau\ge v_0$, then  
putting \eqref{3.14} and \eqref{3.15} together we get
$$
\mu_u(\tau)- \th \mu_v(\tau)\le 0.
$$
For $\tau< v_0$,  it follows clearly that
\begin{align*}
\mu_u(\tau)\le |\Omega|=\th|\Omega^\sharp|= \th\mu_v(\tau),
\end{align*}
thus 
$$
\mu_u(\tau)- \th \mu_v(\tau)\le 0
$$
holds for all $\tau>0$,  hence inequality \eqref{1.8} holds. 

If inequality \eqref{1.8} holds as an equality, then  inequality \eqref{3.14} holds as an equality as well, which implies
$\Omega_\tau\subset\subset \Omega
$, and isoperimetric inequality \eqref{2.6} holds as an equality for $\Omega_\tau$. Hence $\Omega_\tau$ is a geodesic ball in $M_\k$ for $\tau\ge v_0$. So $\Omega$ is a round geodesic ball in $M_\k$, $\th=1$, and $u(x)$ is a radial function. Thus we conclude  $\b(x)$ is a constant on $\p \Omega$ and $M$ is isometric to $M_\k$. We complete the proof of Theorem \ref{thm2}.
\end{proof}

\section*{Acknowledgments} {
We thank Professor Ying Zhang for
his encouragement and support, and
 are grateful to  Professors Daguang Chen and  Carlo Nitsch for some helpful discussions and  comments. We also would like to thank the anonymous referees for catching some typos, which improved the readability of this paper.}

\bibliographystyle{plain}
\bibliography{ref}

\end{document}